\documentclass[12pt,a4paper]{amsart}
\usepackage{amsmath,amsthm}
\usepackage{amssymb}
\usepackage[arrow,matrix]{xy}
\usepackage{amsmath,amsfonts,amssymb,amscd,verbatim,delarray}
\usepackage{graphicx}

\theoremstyle{plain}
\newtheorem{theorem}{Theorem}[section]

\newtheorem{corollary}[theorem]{Corollary}\newtheorem{Corollary}[theorem]{Corollary}
\newtheorem{Proposition}[theorem]{Proposition}
\newtheorem{Lemma}[theorem]{Lemma}

\theoremstyle{definition}
\newtheorem{Definition}[theorem]{Definition}

\newtheorem{remark}[theorem]{Remark}
\newtheorem{example}[theorem]{Example}\newtheorem{Example}[theorem]{Example}

\newtheorem{properties}[theorem]{Properties}

\newcommand{\thmref}[1]{Theorem~\ref{#1}}
\newcommand{\secref}[1]{Section~\ref{#1}}

\newcommand{\lemref}[1]{Lemma~\ref{#1}}
\newcommand{\corref}[1]{Corollary~\ref{#1}}

\newcommand{\propref}[1]{Proposition~\ref{#1}}
\newcommand{\Propertref}[1]{Properties~\ref{#1}}

\DeclareMathOperator{\PSL}{PSL} \DeclareMathOperator{\SL}{SL}
 \DeclareMathOperator{\GL}{GL}

\newcommand{\BC}{\mathbb{C}}

\newcommand{\BQ}{\mathbb{Q}}
\newcommand{\BA}{\mathbb{A}}
                                              \newcommand{\LA}{\mathcal{L}}

\newcommand{\al}{\alpha}
\newcommand{\be}{\beta}

\newcommand{\vp}{\varphi}

\def\Q{{\mathbb Q}}
\def\Z{{\mathbb Z}}
\def\C{{\mathbb C}}
\def\R{{\mathbb R}}

                                                                     \def\BG{{\mathbb G}}

                             \def\ev{\mathrm{ev}}

 \def\bx{\mathbf{x}}
\def\bu{\mathbf{u}}

\def\GL{\mathrm{GL}}
\def\SL{\mathrm{SL}}

\def\PSL{\mathrm{PSL}}
                                           \def\Lie{\mathrm{Lie}}

                                                     \def\End{\mathrm{End}}

                                             \def\Aut{\mathrm{Aut}}

                    \def\sll{\mathfrak{sl}}

\begin{document}

\title  {Surjectivity of certain  word maps  on $\PSL(2,\BC)$  and $\SL(2,\BC)$}
\author{Tatiana Bandman }
\address{Bandman: Department of Mathematics, Bar-Ilan University, 5290002 Ramat Gan, ISRAEL}
\email{bandman@macs.biu.ac.il}

\author{Yuri G. Zarhin}
\address{Zarhin: Department of Mathematics, Pennsylvania State University,
University Park, PA 16802, USA}
\email{zarhin\char`\@math.psu.edu}





\begin{abstract} Let $n \ge 2$ be an integer and $F_n$ the free group on $n$
 generators, $F^{(1)}, F^{(2)}$ its first and second
derived subgroups. Let $K$ be an algebraically closed field of characteristic zero. We show that if $w\in F^{(1)}\setminus F^{(2)},$
then the corresponding word map
 $\PSL(2,K)^n\to \PSL(2,K)$ is surjective. We also describe certain words maps that are surjective on $\SL(2,\BC).$
\end{abstract}
\maketitle

\section{Introduction}

The surjectivity of word maps on groups became recently  a vivid topic: the  review on the latest activities  may be found in
\cite{Se}, \cite{Ku}, \cite {BGaK}, \cite{KBKP}.

Let  $w\in F_n$ be an element of the free group $F_n$ on $n>1$ generators
$g_1,\dots,g_n:$
$$w =\prod\limits_{i=1}^{k}g_{n_i}^{m_i}, \ 1\le{n_i}\le n.$$

For a   group  $G$  by the same letter $w$ we shall denote  the
corresponding word map $w:G^n\to G$   defined as a non-commutative product by the formula
\begin{equation}\label{wordn}
w(x_1,\dots, x_n)=\prod \limits_{i=1}^{k}x_{n_i}^{m_i}.\end{equation}

We call $w (x_1,\dots, x_n)$  both   {\it a word in $n$ letters } if considered as an element of a free group
and   {\it a word map in $n$ letters } if considered as the corresponding   map  $G^n\to G.$

We assume that  it is reduced, i.e. ${n_i}\ne{n_{i+1}}$  for every $1\le i\le k-1$ and $m_i\ne 0$ for $1\le i\le k. $

Let $K$  be a field and
$H$  a connected semisimple algebraic linear group. If $w$ is not the identity then  by  Theorem of A Borel  (\cite{Bo})
the regular map of (affine) $K$-algebraic varieties
$$w: H^n \to H, \ (h_1,\dots,h_n) \mapsto w(x_1,\dots,x_n)$$
is {\sl dominant}, i.e., its image is a Zariski dense  subset of $H$.
Let us consider the group  $G=H(K)$ and
the image $$w_G:=w(G^n)=\{z\in G \mid z=w(x_1,\dots,x_n) \text{ for some } \ (x_1,\dots,x_n)\in G^n\}.$$
We  say that a word (a word map) $w$ is {\sl surjective} on
 $G$ if $w_G=G.$

In  \cite{KKMP}, \cite{Ku} formulated is the following Question.

{\bf Problem 7 of \cite{KKMP},  Question 2.1 (i) of \cite{Ku}.}
Assume that $w $ is not a  power of another reduced word and $G=H(K)$  a
 connected semisimple algebraic linear group.

Is $w$ surjective when $K = \BC$ is a field of complex numbers and $H$ is of adjoint type?

According  to\cite  {Ku}, Question 2.1(i) is still open, even in the simplest case $G =\ PSL(2,\BC),$
even for words in two letters.

We consider  word  maps   on  groups $G=\SL(2,K)$ and $\tilde G=\PSL(2,K).$
Put
$$F:=F_n, \  F^{(1)}=[F,F], \ F^{(2)}=[F^{(1)},F^{(1)}] .$$
As usual, $\Z,\Q,\R,\C$ stand for the ring of integers and fields of rational, real and complex numbers respectively.  $\BA(K)^m_{x_1,\dots,x_m}$
or, simply, $\BA^m,$  stands for the $n-$dimensional  affine space over a field $K$ with coordinates
$x_1,\dots,x_m.$ If  $K=\BC$, we use $\BC^m_{x_1,\dots,x_m}.$

 Let $w\in F$. For a corresponding word map on $G=\SL(2,K)$
we check the following properties of the image  $w_G$.
\begin{properties}\label{properties}\hskip10 cm
 \begin{description}
\item[a] $w_G$ contains all semisimple elements $x $  with $tr(x)\ne 2;$
\item[b] $w_G$ contains all unipotent elements $x$ with $tr(x)= 2;$
\item[c] $w_G$ contains all  minus unipotent elements $x $ with $tr(x)= -2$ and  $ x\ne -id;$
\item[d] $w_G$ contains  $ -id.$
\end{description}
\end{properties}
 The word map $w$ is surjective on $G=\SL(2,K)$
if all  \Propertref{properties} are met.
For surjectivity on $\tilde G=\PSL(2,K)$ it is sufficient that only  \Propertref{properties}  {\bf a} and
{\bf b} are valid.

\begin{Definition} (cf.\cite{BGG}) We say that the word map $w$ is almost surjective
on $G=\SL(2,K)$ if it has \Propertref{properties} {\bf a},{\bf b}, and {\bf c}, i.e
$w_G\supset \SL(2,K)\setminus -\{id\}.$\end{Definition}

 The goal of the paper is to  describe certain  words $w\in F$ such that the corresponding   word maps
are surjective or almost surjective  on $G$ and/or
 $\tilde G.$

Assume that the field $K$ is algebraically closed. If $w(x_1,\dots,x_d)=x_i^n$ then $w$ is surjective on $G$ if and only if $n$ is odd  (see (\cite{Ch1}, \cite{Ch2}). Indeed,
 the element
$$x=\begin{pmatrix}-1&1\\0&-1\end {pmatrix}$$
is not a square in $\SL(2,K).$  Since only the elements with $tr(x)=-2$  may be outside $w_G$
(\cite{Ch1}, \cite{Ch2}),
the  induced by $w$  word map $\tilde w$ is surjective on $\tilde G.$

Consider  a word map \eqref{wordn}.  For an index $j\le n$ let
$S_j=\sum\limits_{n_i=j}m_i.$

If,  say,  $S_1\ne 0, $ then $w(x_1,id,\dots,id)=x_1^{S_1}, $ hence
word $w$ is surjective on $\PSL (2,K). $  If $S_j=0$ for all $1\le j\le d,$
then   $w\in F^{(1)}=[F,F].$
In \secref{wrd} we prove  (see \corref{joint})
the following
\begin{theorem}\label{1.3} The word map defined by a  word $w\in  F^{(1)}\setminus F^{(2)}$ is surjective on $\PSL(2,K)$  if $K$ is an algebraically closed field with $char(K)=0.$\end{theorem}

The proof makes use of  a variation on the Magnus Embedding Theorem, which is stated in  \secref{magnus}
and proven in  \secref{prf}.

In  \secref{section2}, \secref{section3}, and  \secref{section4} we  consider words in two variables, i.e. $n=2.$
In this case we give explicit formulas  for $w(x,y),$  where $x,y\in\SL(2,\BC)$  are upper triangular matrices. Using   explicit formulas
 in  \secref{section3} and  \secref{section4} we
provide   criteria for surjectivity and almost surjectivity of a word map on $G=\SL(2,\BC).$
In \secref{section3} these criteria  are  formulated in terms of properties of exponents  ${a_i}, {b_i}, \ i=1\dots,k ,$ of a word
\begin{equation}\label{word}w(x,y)=\prod \limits_{i=1}^{k}x^{a_i}y^{b_i},\end {equation}
where $a_i\ne 0 $ and $b_i\ne 0, $ {for all }$ i=1,...,k.$
 A sample of such criteria is
\begin{corollary}\label{1.4}
If all $b_i$ are positive, then the word map  $w$ is either surjective
or the square of another word $v\ne id$.\end{corollary}

In  \secref{section4} we connect the almost surjectivity of a word map with a property  of the corresponding trace map.
The last sections  contain explicit examples.

{\bf Acknowledgments}

We thank
 Boris Kunyavskii for inspiring questions
and  useful comments, and  Eugene Plotkin,   Vladimir L. Popov,  and Alexander Premet for help with references.

We are grateful to a referee  for suggesting to use the Magnus  Embedding Theorem.
T. Bandman is grateful to both referees  for pointing out
 the inaccuracies of the first version of the paper.

T. Bandman was
partially supported by the Ministry
of Absorption (Israel), the Israeli Science Foundation (Israeli
Academy of Sciences, Center of Excellence Program), and the Minerva
Foundation (Emmy Noether Research Institute of Mathematics).
 This work was partially done while
the author was visiting Max Planck Institute of Mathematics at Bonn.
The hospitality of this Institution is greatly appreciated.

This work  of Yu. G. Zarhin  was partially supported by a grant from the Simons Foundation
(\#246625 to Yuri Zarkhin).
This paper was written in May--June 2015 when Yu. Zarhin  was a visitor at  Department of Mathematics of
the Weizmann Institute of Science (Rehovot, Israel), whose hospitality is gratefully acknowledged.

\section{Semisimple elements} \label{semi-simple}

 Let $K$ be an algebraically closed field with $char(K)=0,$  and  \  $G=\SL(2,K).$
 Consider a word map $w:G^n\to G:$
$$w(x_1,\dots, x_n)=\prod \limits_{i=1}^{k}x_{n_i}^{m_i}.$$
We consider
$G$ as  an affine set
$$G=\{ad-bc=1\}\subset \BA^4_{a,b,c,d}.$$

The following Lemma is , may be , known, but the authors do not have a proper reference.

\begin{Lemma} \label{nonomitting}
A regular non-constant function on $G^n$ omits no values in $K.$\end{Lemma}
\begin{proof} Since all the sets are affine, a function  $f$ regular on $G^k$ is a restriction of a polynomial $P_f$ onto  $G^k.$
We use induction on $k.$

{\bf Step 1. } $k=1. $  $$G=\{ad-bc=1\}\subset \BA^4_{a,b,c,d}$$
is an irreducible quadric.
Assume that $f\in  K[G]$ omits a value.
Let $p:G\to\BA^1_{a}$ be a projection defined by $p(a,b,c,d)=a.$
If $a\ne 0$ then fiber $F_a:=p^{-1}(a)\cong\BA^2_{b,c},$  is an affine space with coordinates $b,c$
because $d=\frac{1+bc}{a}.$  Since $f$ omits a value, the restriction $ f\bigm |_{F_a}$ is  constant  for every $a\ne 0.$
`Therefore it is constant on every fiber  ( note that the fiber $a=0$ is conneceted).
On the other hand, $f$ has to be constant along the curve
$$C=\{(a,0,1,1)\}\cong \BA^1_a(K). $$
 Since curve $C\subset G$  intersects every fiber  $F_a$ of projection $p,$
function  $f$ is constant on $G.$

{\bf Step 2}.  Assume that the statement of the Lemma is valid for all $k\le n.$
Let $f\in K[G^n]$ omit a value.   We have:  $G^n=M\times N, $ where $M=G^{n-1}$ and $N=G.$
Let $p:G^n\to N $ be a natural projection. Then, by induction assumption,  $f$ is constant along every
fiber of this projection. Take $x\in M $ and consider the set $C=x\times N\subset G^n.$   Then$ f\bigm |_{C}=const$
and $C$  intersects every fiber of $p.$  Hence,
$ f$ is constant.
\end{proof}

\begin{Proposition}\label{semisimple}
For every  word $w(x_1,\dots, x_k)\ne id$
the image $w_G$
 contains every  element $z\in G$ with $a:=tr(z)\ne \pm 2.$ \end{Proposition}

\begin{proof} We consider $G^n\subset\BA(K)^{4n}$ as the product  ($1\le i\le n$) of
$$G_i=\{a_id_i-b_ic_i=1\}\subset \BA^4_{a_i,b_i,ic_i,d_i}.$$

The function $f(a_1,b_1,c_1,d_1,\dots, a_n,b_n,c_n,d_n)=tr(w(x_1,\dots,x_n))$
is a polynomial  in $4n$ variables with integer coefficients, i.e $f\in K[G^n].$
According to \lemref{nonomitting}, it takes on all the values in $K.$

Thus for every value $A\in K$ there is
element $u= w(y_1,\dots,y_n)\in w_G$
such that $tr(u)=A.$

Let now $z\in G, \  A:=tr(z)\ne\pm 2. $  Since $tr(z)=tr(u),$  $z$ is conjugate to $u,$
i.e there is $v\in G$ such that $vuv^{-1}=z$. Hence
$$z=w(vy_1v^{-1},\dots,vy_nv^{-1}).$$\end{proof}

It follows that in order to check whether the word map $w$ is surjective on $G$ (or on $\tilde G$) it
is
sufficient to check whether the elements $z$ with $tr(z)=\pm 2$
(or the elements $z$ with $tr(z)= 2,$ respectively) are in the image.
For that we need a version of the Embedding Magnus Theorem.

\section{ Variation on  Magnus Embedding Theorem: Statements}\label{magnus}

Let $n \ge 2$ be an integer and $\Lambda_n=\Z[t_1, t_1^{-1}, \dots , t_n, t_n^{-1}]$ be the ring of Laurent polynomials in $n$ independent variables
$t_1, \dots , t_n$ over $\Z$.
Let  $F=F_n$ be a free group of rank $n$ with generators $\{g_1, \dots , g_n\}$.  Recall:  we write $F^{(1)}$ for the derived subgroup of $F$ and
$F^{(2)}$ for the derived subgroup of $F^{(1)}$.  We have
$$F^{(2)}\subset F^{(1)}\subset F;$$
both $F^{(1)}$ and $F^{(2)}$ are normal subgroups in $F$. The quotient $A:=F/F^{(1)}= \Z^n$ is a free abelian group of rank $n$ with (standard) generators $\{e_1, \dots , e_n\}$ where each $e_i$ is the image of $g_i$ ($1 \le i \le n$). It is well known that the  group ring $\Z[A]$ of $A$ is canonically isomorphic to $\Lambda_n$: under this isomorphism each
$$e_i \in A\subset \Z[A]$$
 goes to
$$t_i \in \Z[t_1, t_1^{-1}, \dots , t_n, t_n^{-1}]=\Lambda_n.$$
We write $R_n$ for the ring of polynomials
$$\Lambda_n[s_1, \dots , s_n]= \Z[t_1, t_1^{-1}, \dots , t_n, t_n^{-1}; s_1, \dots , s_n]$$
in $n$ independent variables $s_1, \dots , s_n$ over $\Lambda_n$.
If $R$ is a commutative ring with $1$ then we write $T(R)$ for the group of invertible $2\times 2$ matrices of the form
$$\begin{bmatrix}
a&0\\
b&1
\end{bmatrix}$$
with $a \in R^{*}, b \in R$
and $ST(R)$  for the group of unimodular $2\times 2$ matrices of the form
$$\begin{bmatrix}
a&0\\
b&a^{-1}
\end{bmatrix}$$
with $a \in R^{*}, b \in R$. We have
$$T(R)\subset \GL(2,R), \ ST(R) \subset \SL(2,R).$$
Every homomorphism $R \to R^{\prime}$ of commutative rings (with 1) induces the natural group homomorphisms
$$T(R)\to T(R^{\prime}), \ ST(R)\to ST(R^{\prime}),$$
which are injective if $R \to R^{\prime}$ is injective.

The following assertion (that is based on the properties of the famous Magnus embedding  \cite{MagnusCrelle})  was proven in \cite[Lemma 2]{Wehrfritz}.

\begin{theorem}
\label{W}
The assignment
$$g_i \mapsto \begin{bmatrix}
t_i&0\\
s_i&t_i^{-1}
\end{bmatrix} \ (1 \le i \le n)$$
extends to a group homomorphism
$$\mu_W: F \to ST(\Lambda_n)$$ with kernel $F^{(2)}$ and therefore defines an embedding
$$F/F^{(2)} \hookrightarrow ST(R_n)\subset \SL(2,R_n).$$
\end{theorem}

It follows from Theorem \ref{W} that if $K$ is a field of characteristic zero, whose transcendence degree over $\Q$ is, at least, $2n$ then there is an embedding
$$F/F^{(2)} \hookrightarrow ST(K)\subset \SL(2,K).$$
(In particular, it works for $K=\R$, $\C$ or the field $\Q_p$ of $p$-adic numbers \cite{Wehrfritz}.)
 The aim of the following considerations is   to replace in this statement the lower bound $2n$ by $n$.

\begin{theorem}
\label{main}
The assignment
$$g_i \mapsto \begin{bmatrix}
t_i&0\\
1&t_i^{-1}
\end{bmatrix} \ (1 \le i \le n)$$
extends to a group homomorphism
$$\mu_1: F \to ST(\Lambda_n)$$ with kernel $F^{(2)}$ and therefore defines an embedding
$$F/F^{(2)} \hookrightarrow ST(\Lambda_n)\subset \SL(2,\Lambda_n).$$
\end{theorem}

\begin{remark}
Let
$$\ev_1: R_n=\Lambda_n[s_1, \dots , s_n] \to \Lambda_n$$
be the $\Lambda_n$-algebra homomorphism that sends all $s_i$ to $1$ and let
$${\ev_1}^{*}: ST(R_n) \to ST(\Lambda_n)$$
be the group homomorphism induced by $\ev_1$.  Then  $\mu_1$ coincides with the composition
$${\ev_1}^{*} \mu_W: F \to ST(R_n) \to ST(\Lambda_n).$$
\end{remark}

\begin{corollary}
\label{degree}
 Let $K$ be a field of characteristic zero. Suppose that the transcendence degree of $K$ over $\Q$ is, at least, $n$.  Then there is a group embedding
$$F/F^{(2)} \hookrightarrow ST(K)\subset \SL(2,K).$$
\end{corollary}

Proof of Theorem \ref{main} is based on the following observation.
\begin{Lemma}
\label{magnusE}
Let $K$ be a field of characteristic zero. Suppose that the transcendence degree of $K$ over $\Q$ is, at least, $n$ and let $\{u_1, \dots , u_n\}\subset K$ be an $n$-tuple of algebraically independent elements (over $\Q$). Let $\Q(u_1, \dots , u_n)$ be the subfield of $K$ generated by $\{u_1, \dots , u_n\}$ and let us consider $K$ as the  $\Q(u_1, \dots , u_n)$-vector space.
Let $\{y_1, \dots , y_n\}\subset K$  be a $n$-tuple that is linearly independent over  $\Q(u_1, \dots , u_n)$. Let $R$ be the subring of $K$ generated
by $3n$ elements $u_1, u_1^{-1}, \dots , u_n, u_n^{-1}; y_1, \dots, y_n$.

Then the assignment
$$g_i \mapsto \begin{bmatrix}
u_i&0\\
y_i&1
\end{bmatrix} \ (1 \le i \le n)\in T(R)$$
extends to a group homomorphism
$$\mu: F \to T(R)\subset T(K)$$ with kernel $F^{(2)}$ and therefore defines an embedding
$$F/F^{(2)} \hookrightarrow T(R)\subset  T(K).$$
\end{Lemma}

\begin{example}
\label{basicE}
{\rm Let $K$ be the field $\Q(t_1, \dots , t_n)$ of rational functions in $n$ independent variables $t_1, \dots , t_n$ over $\Q$. One may view $\Lambda_n$
as the subring of $K$ generated by $2n$ elements $t_1, t_1^{-1}, \dots , t_n, t_n^{-1}$. By definition, the
 $n$-tuple
$\{t_1, \dots , t_n\}\subset K$  is algebraically independent (over $\Q$). Clearly, the  $n$-tuple
$$\{u_1=t_1^2, \dots , u_i=t_i^2, \dots,  u_n=t_n^2\}\subset K$$
 is also algebraically independent. Then the $n$  elements
$$y_1=t_1, \dots, y_i=t_i, \dots, y_n=t_n$$
 are linearly independent over the (sub)fileld $\Q(t_1^2, \dots , t_n^2)=\Q(u_1, \dots , u_n)$. Indeed, if a rational function
$$f(t_1, \dots, t_n)=\sum_{i=1}^n t_i\cdot f_i$$
where all $f_i \in \Q(t_1^2, \dots , t_n^2)$ then
$$2 t_1 f_1=f(t_1, t_2, \dots, t_n)-f(-t_1, t_2, \dots, t_n), \dots,$$
$$ 2t_i f_i=f(t_1, \dots,t_i, \dots,  t_n)-f(t_1, \dots, -t_i, \dots, t_n), \dots,$$
$$ 2t_n f_n=f(t_1, \dots,t_i, \dots,  t_n)-f(t_1, \dots, t_i, \dots, -t_n).$$
This proves that if $f=0$ then all $f_i$ are also zero, i.e., the set $\{t_1, \dots, t_n\}$ is linearly independent over $\Q(t_1^2, \dots t_n^2)$.

By definition,  $R$ coincides with the subring of $K$ generated by  $3n$ elements
$$t_1^2, t_1^{-2}, \dots , t_n^2, t_n^{-2}; t_1, \dots , t_n.$$
This implies easily that $R=\Lambda_n$. } Applying Lemma \ref{magnusE}, we conclude
the Example by the following statement.

{\sl  The assignment
$$g_i \mapsto \begin{bmatrix}
t_i^2&0\\
t_i&1
\end{bmatrix} \ (1 \le i \le n)\in T(\Lambda_n)$$
extends to a group homomorphism
$$\mu: F \to T(R) =T(\Lambda_n)$$
 with kernel $F^{(2)}$ and therefore defines an embedding
$$F/F^{(2)} \hookrightarrow  T(\Lambda_n).$$}
\end{example}

We prove Lemma \ref{magnusE}, Theorem \ref{main} and Corollary \ref{degree} in Section \ref{prf}.

\section{ Variation on  the Magnus Embedding Theorem: Proofs}
\label{prf}

\begin{proof}[Proof of Lemma \ref{magnusE}]
Let
$$\Lambda \subset \Q(u_1, \dots , u_n)\subset K$$
be the subring generated by $2n$ elements $u_1, u_1^{-1}, \dots , u_n, u_n^{-1}$.
Since $u_i$ are algebraically independent over $K$, the assignment
$$t_i \mapsto u_i, \ t_i^{-1} \mapsto u_i^{-1}$$
 extends to a ring isomorphism
$\Lambda_n \cong \Lambda$.  The linear independence of $y_i$'s over $\Q(u_1, \dots , u_n)$ implies that
$M=\Lambda \cdot y_1 + \dots + \Lambda \cdot y_n \subset R\subset K$
is a free $\Lambda$-module of rank $n$.
On the other hand, let
$$U \subset \Lambda^{*}\subset \Q(u_1, \dots , u_n)^{*}\subset K^{*}$$
be the multiplicative (sub)group generated by all $u_i$. The assignment
$g_i\mapsto u_i$ extends to the surjective group  homomorphism
$$\delta: F \twoheadrightarrow U$$
with kernel $F^{(1)}$ and gives rise to the group isomorphism
$$A \cong U,$$
which sends $e_i$ to $u_i$ and allows us to identify the group ring $\Z[U]$ of $U$ with $\Lambda\cong \Lambda_n=\Z[A]$.
Notice that $M$ carries the natural structure of free $\Z[U]$-module of rank $n$ defined by
$$\lambda(m): =\lambda \cdot m \in K \ \forall \lambda \in \Z[U]=\Lambda\subset K,  m \in M\subset K.$$
We have
$$\mu(F)\subset \begin{bmatrix}
U&0\\
M&1
\end{bmatrix} \subset T(R)\subset \GL_2(R).$$
It follows from  \cite[Lemma 1(c) on p. 175]{Wilson} that $\ker(\mu)$ coincides with the derived subgroup of $\ker(\delta)$.
Since $\ker(\delta)=F^{(1)}$, we conclude that $\ker(\mu)=F^{(2)}$ and we are done.
\end{proof}

\begin{proof}[Proof of Theorem \ref{main}]
Let us return to the situation of Example \ref{basicE}. In particular, the group embedding (we know that it is an embedding, thanks to already proven Lemma \ref{magnusE})
$$\mu: F \hookrightarrow T(\Lambda_n)\subset \GL_2(\Lambda_n)$$
is defined by
$$\mu(g_i) = \begin{bmatrix}
t_i^2&0\\
t_i&1
\end{bmatrix} \ \in T(\Lambda_n)$$
for all $g_i$.

Let  us consider the group homomorphism
$$\kappa: F \to \Lambda_n^{*}, \ g_i \mapsto t_i.$$
Since $t_i$ are algebraically independent, they are multiplicatively independent and
$$\ker(\kappa)=F^{(1)}.$$
I claim that $\mu_1: F \to ST(\Lambda_n)$ coincides with
 the group homomorpism
$$g \mapsto \kappa(g)^{-1}\cdot \mu(g).$$
Indeed, we have for all $g_i$
$$\kappa(g_i)^{-1}\cdot\mu(g_i)=t_i^{-1}\cdot  \begin{bmatrix}
t_i^2&0\\
t_i&1
\end{bmatrix} = \begin{bmatrix}
t_i&0\\
1& t_i^{-1}
\end{bmatrix}
=\mu_1(g_i)
\subset ST(\Lambda_n),$$
which proves our claim. Recall that we
 need to check that $\ker(\mu_1)=F^{(2)}$. In order to do that, first notice that
$\mu_1(g)$ is of the form
$\begin{bmatrix}
\kappa(g)&0\\
*& \kappa(g)^{-1}
\end{bmatrix}$
for all $g \in F$
just because it is true for all $g=g_i$. This implies that
$$\ker(\mu_1)\subset \ker(\kappa)=F^{(1)}.$$
But $\mu=\mu_1$ on $F^{(1)}$. This implies that
$$\ker(\mu_1)=\ker(\mu)\bigcap F^{(1)}.$$
However, as we have seen in Example \ref{basicE},
$$\ker(\mu)= F^{(2)}\subset F^{(1)}.$$
This implies that
$$\ker(\mu_1)=F^{(2)}\bigcap F^{(1)}=F^{(2)}$$
and we are done.
\end{proof}

\begin{proof}[Proof of Corollary \ref{degree}]
There exists an $n$-tuple $\{x_1, \dots, x_n\}\subset K$ that is algebraically independent over $\Q$. The assignment
$$t_i \mapsto x_i, \ t_i^{-1} \mapsto x_i^{-1}$$
 extends to an {\sl injective} ring homomorphism
$$\Lambda_n= \Z[t_1, t_1^{-1}, \dots , t_n, t_n^{-1}]\hookrightarrow K.$$
This implies that $ST(\Lambda_n)$ is isomorphic to a subgroup of $ST(K)$.
Thanks to Theorem \ref{main}, $F/F^{(2)}$ is isomorphic to a subgroup of $ST(\Lambda_n)$. This implies that $F/F^{(2)}$ is isomorphic to a subgroup of
$ST(K)$.
\end{proof}

{\bf Remark}. Similar arguments prove the following generalization of Theorem \ref{main}.

\begin{theorem}
Let  $\{b_1, \dots, b_n\}$ be an $n$-tuple of nonzero integers. Then
the assignment
$$g_i \mapsto \begin{bmatrix}
t_i&0\\
b_i&t_i^{-1}
\end{bmatrix} \ (1 \le i \le n)$$
extends to a group homomorphism
$F \to ST(\Lambda_n)$ with kernel $F^{(2)}$.
\end{theorem}

\section{Word maps and unipotent elements}
\label{wrd}

\begin{Lemma}\label{Lw1}
Let $w$ be an element of $F^{(1)}$ that does not belong to $F^{(2)}$. Then there exists a nonzero Laurent polynomial
$$\LA_w=\LA_w(t_1, \dots t_n)\in \Z[t_1, t_1^{-1}, \dots , t_n, t_n^{-1}]=\Lambda_n$$
such that
$$\mu_1(w)=
\begin{bmatrix}
1 &0\\
\LA_w& 1
\end{bmatrix}.$$
\end{Lemma}

\begin{proof}
We have seen in the course of the proof of Theorem \ref{main} that for all $g \in F$
$$\mu_1(g)=
\begin{bmatrix}
\kappa(g)&0\\
*& \kappa(g)^{-1}
\end{bmatrix} \in ST(\Lambda_n).$$
This means that there exists a Laurent polynomial $\LA_g \in \Lambda_n$ such that
$$
\mu_1(g)=
\begin{bmatrix}
\kappa(g)&0\\
\LA_g& \kappa(g)^{-1}
\end{bmatrix}.
$$
We have also seen that if $g\in F^{(1)}$ then $\kappa(g)=1$.
 Since $w\in F^{(1)}$,
$$
\mu_1(w)=
\begin{bmatrix}
1&0\\
\LA_w& 1
\end{bmatrix}
$$
with $\LA_w \in \Lambda_n$. On the other hand, by Theorem \ref{main}, $\ker(\mu_1)=F^{(2)}$.
Since $w\not\in F^{(2)}$, $\LA_w \ne 0$ in $\Lambda_n$.
\end{proof}

\begin{corollary}
\label{value}
Let $w$ be an element of $F^{(1)}$ that does not belong to $F^{(2)}$.
 Suppose that  ${\mathbf a}=\{a_1, \dots, a_n\}$ is an $n$-tuple of nonzero rational numbers
 such that
 $$c:=\LA_w(a_1, \dots, a_n) \ne 0.$$
  (Since $\LA_w \ne 0$, such an $n$-tuple always exists.) Let us consider the group homomorphism
$$\gamma_{\mathbf a}: F \to ST(\Q)\subset \SL(2,\Q), \ g_i \mapsto
 \begin{bmatrix}
a_i&0\\
1& a_i^{-1}
\end{bmatrix}:=Z_i.
$$
 Then
 $$\gamma_{\mathbf a}(w)=
  \begin{bmatrix}
1&0\\
c& 1
\end{bmatrix}=w(Z_1,\dots,Z_n).$$
is a unipotent matrix that is not the identity matrix.
\end{corollary}

\begin{proof}
One has only to notice that $\gamma_{\mathbf a}$ is the composition of $\mu_1$ and the homomorphism $ST(\Lambda_n) \to ST(\Q)$ induced by the ring homomorphism
$$\Lambda_n \to \Q, \ t_i \mapsto a_i, t_i^{-1} \mapsto a_i^{-1}.$$
\end{proof}

\begin{corollary}
\label{uni}
Let $w$ be an element of $F^{(1)}$ that does not belong to $F^{(2)}$.
Let $K$ be a field of characteristic zero. Then for every unipotent matrix
$X \in \SL(2,K)$ there exists  a group homomorphism $\psi_{w,X}:F \to \SL(2,K)$ such that
$$\psi_{w,X}(w)=X.$$  In other words, there exist $Z_1,\dots,Z_n\in \SL(2,K)$ such that $w(Z_1,\dots,Z_n)=X.$
\end{corollary}

\begin{proof}
We have
$$\Q \subset K, \ \SL(2,\Q)\subset \SL(2,K)\vartriangleleft \GL(2,K).$$
We may assume that $X$ is {\sl not} the identity matrix. Let ${\mathbf a}=\{a_1, \dots, a_n\}$ and $\gamma_{\mathbf a}$ be as in Corollary \ref{value}. Recall that
 $c=\LA_w(a_1, \dots, a_n) \ne 0$. Then there exists a matrix $S \in \GL(2,K)$ such that
$$X= S
\begin{bmatrix}
1&0\\
c& 1
\end{bmatrix}
S^{-1}.$$
Let us consider the group homomorphism
$$\psi_{w,X}:F \to \SL(2,K), \ g \mapsto S \gamma_a(g) S^{-1}.$$
Then $\psi_{w,X}$ sends $w$ to

\begin{equation}\label{explicit}S \gamma_{\mathbf a}(w) S^{-1}= S
\begin{bmatrix}
1&0\\
c& 1
\end{bmatrix}
S^{-1}
=X.\end{equation}
\end{proof}

\begin{corollary}\label{joint} (\thmref {1.3})
Let $w$ be an element of $F^{(1)}$ that does not belong to $F^{(2)}$.
Let $K$ be an algebraically closed  field of characteristic zero. Then the word map $w$
is surjective on $\PSL(2,K).$
\end{corollary}

\begin{proof} Consider $w$ as a word map on $G=\SL(2,K).$
Due to \corref{uni}  the image  $w_G$ contains all unipotents.
According to \propref{semisimple} the image contains all the
 semisimple elements as well.  Thus, the word map $w$ has the \Propertref{properties} {\bf a} and
  {\bf b}. It follows that it is surjective on  $\PSL(2,K).$
 \end{proof}

\begin{remark} In \cite{ET}  the words from $F^{(1)}\setminus F^{(2)}$ are  proved to be surjective on
$SU(n)$   for an infinite set of integers $n.$ \end{remark}

\begin{theorem}
 Let $w$ be an element of $F^{(1)}$ that does not belong to $F^{(2)}$. Let  $G$ be a connected semisimple linear algebraic group over a field
 $K$   of characteristic zero. If $u \in G(K)$ is a unipotent element
 then there exists  a group homomorphism $F \to G(K)$ such that the image of $w$ coincides with $u$. In other words, there exist $Z_1,\dots,Z_n\in G(K)$ such that $w(Z_1,\dots,Z_n)=u.$
 \end{theorem}

 \begin{proof}
Let ${\mathbf a}=\{a_1, \dots, a_n\}$, $\gamma_{\mathbf a}$ and
 $c=\LA_w(a_1, \dots, a_n) \ne 0$ be as in  Corollary \ref{value}.
By Lemma \ref{GJM} below, there exists
 a group  homomorphism $\phi: ST(K) \to G(K)$
 such that $u=\phi(\bu_1)$ for
 $$\bu_1=\begin{bmatrix}
1&0\\
c& 1
\end{bmatrix}\in ST(K).$$
Now the result follows from Corollary \ref{value}: the desired homomorphism is the composition
$$\phi \ \gamma_{\mathbf a}: F \to ST(K) \to G(K).$$
 \end{proof}

\begin{Lemma}
\label{GJM}
Let $K$ be a field of characteristic zero,  $G$ a connected semisimple linear algebraic $K$-group of positive dimension,
 and $u$ a unipotent element of $G(K)$.
Then for every nonzero $c\in K$ there is a group homomorphism
$\phi: ST(K) \to G(K)$ such that $u$ is the image of
$$\bu_1=
 \begin{bmatrix}
1&0\\
c& 1
\end{bmatrix}
\in ST(K).
$$
\end{Lemma}
\begin{proof}

Let us identify  the additive algebraic $K$-group $\BG_a$ with  the closed subgroup  $H$  of all matrices of the form $ v(t)=\begin{bmatrix}
1&0\\
t& 1
\end{bmatrix}
$ in $\SL(2)$. Its Lie subalgebra $\Lie(H)$ is the one-dimensional   $K$-vector subspace
$\Lie(H)=\{\lambda \bx_0 \ |\lambda \in K\}$ of $ \sll_2(K)$ generated by the matrix
$$\bx_0=  \begin{bmatrix}
0&0\\
1& 0
\end{bmatrix}
\subset \sll_2(K).$$
Here we view the $K$-Lie algebra
$\sll_2(K)$ of $2\times 2$ traceless matrices as the Lie algebra of the algebraic $K$-group $\SL(2)$.
 Moreover, $\exp(\lambda \bx_0)=v(\lambda)$ for all $\lambda \in K$.

We may view $G$ as a closed algebraic $K$-subgroup of the matrix group $\GL(N)=\GL(V)$,
 where $V$ is an $N-$dimensional $K$-vector space for a
suitable positive integer $N$. Then
$$u \in G(K)\subset \Aut_{K}(V)=\GL(N,K).$$

Thus  the $K$-Lie algebra $\Lie(G)$ becomes a certain {\sl semisimple} Lie subalgebra of $\End_{K}(V)$.
Here we view  $\End_{K}(V)$
as  the Lie algebra $\Lie(\GL(V))$ of the $K$-algebraic group $\GL(V)$.  As usual, we write
$$\mathrm{Ad}: G(K) \to \Aut_K(\Lie(G))$$
for the adjoint action of $G$. We have
$$\mathrm{Ad}(g)(u)=g u g^{-1}$$
for all
 $$ g \in G(K)\subset \Aut_K(V) \ \mathrm{ and } \ u \in \Lie(G)\subset \End_K(V).$$
Since $u$ is a unipotent element, the linear operator $u-1: V \to V$ is a nilpotent.
 Let us consider  the nilpotent linear operator
$$x=\log(u):= \sum_{i=1}^{\infty}(-1)^{i+1} \frac{(u-1)^i}{i}\in  \End_{K}(V)$$
 (\cite[Sect 7, p. 106]{Borel}, \cite[Sect.23, p. 336]{TaYu})
and the corresponding  homomorphism of algebraic $K$-groups
$$\varphi_u:H \to \GL(V), \ v(t)
 \mapsto \exp(tx)=v(0)+tx+\dots.$$
In particular, since $\bu_1=v(1),$
$$\varphi_u(\bu_1)=u.$$
Clearly, the differential   of $\varphi_u$
$$d\varphi_u: \Lie(H)\to \Lie(\GL(V))=\End_K(V)$$
is  defined as
$$d\varphi_u (\lambda\bx_0)=\lambda x \ \forall \lambda \in K$$
and
 sends $\bx_0$ to $x\in \Lie(\GL(V))$.
Since  $\varphi_u(m)=u^m \in G(K)$ for all integers $m$ and $G$ is closed  in $\GL(V)$ in Zariski topology, the image $\varphi_u(H)$
of $H$ lies in $G$ and therefore one may view $\varphi_u$
as a homomorphism of algebraic $K$-groups
$$\varphi_u: H \to G.$$
This implies that
$$d\varphi_u(\Lie(H)) \subset \Lie(G);$$
in particular,
$x \in \Lie(G)$.

There exists a {\sl cocharacter} 
$$\Phi:\BG_m \to G\subset \GL(V)$$ of $K$-algebraic group $G$ such that for each $\beta \in K^{*}=\BG_m(K)$
$$\mathrm{Ad}(\Phi(\beta))(x)=\beta^2 x$$
(see \cite[Sect. 6, pp. 402--403]{McNinch}. Here $\BG_m$ is the multiplicative algebraic $K$-group.) This means that for all $\lambda \in K$
$$\Phi(\beta)\lambda x \Phi(\beta)^{-1}=\mathrm{Ad}(\Phi(\beta))(\lambda x)=
\lambda\beta^2  x=\beta^2 \lambda x\in \Lie(G)\subset \End_K(V),$$
which
implies that
$$\Phi(\beta)(\exp(\lambda x))\Phi(\beta)^{-1}=\exp\left(\Phi(\beta)\lambda x \Phi(\beta)^{-1}\right)=
\exp(\beta^2 \lambda x).$$
It follows that
$$\Phi(\beta)\left(\exp\left(\frac{\lambda}{c} x\right)\right)\Phi(\beta)^{-1}=\exp\left(\beta^2 \frac{\lambda}{c} x\right).$$
Recall that $ST(K)$ is a {\sl semi-direct product} of its normal subgroup $H(K)$ and the torus
$$T_1(K)=\left\{\begin{bmatrix}
\beta^{-1}&0\\
0& \beta
\end{bmatrix}
, \ \beta \in K^{*}\right\}\subset ST(K).$$
In addition,
$$\begin{bmatrix}
\beta^{-1}&0\\
0& \beta
\end{bmatrix}
\begin{bmatrix}
1&0\\
\lambda& 1
\end{bmatrix}
\begin{bmatrix}
\beta^{-1}&0\\
0& \beta
\end{bmatrix}^{-1}=
\begin{bmatrix}
1&0\\
\beta^2\lambda& 1
\end{bmatrix}
\ \forall \lambda \in K, \beta\in K^{*}.$$
It follows from \cite[Ch. III, Prop. 27 on p. 240]{Bourbaki} that there is a group homomorphism
$$\phi: ST(K) \to G(K)$$
 that  sends each $\begin{bmatrix}
1&0\\
\lambda& 1
\end{bmatrix}$ to $\exp(\frac{\lambda}{c} x)$ and each $\begin{bmatrix}
\beta^{-1}&0\\
0& \beta
\end{bmatrix}$ to $\Phi(\beta)$.
Clearly,
$\phi$ sends $\bu_1=\begin{bmatrix}
1&0\\
c& 1
\end{bmatrix}$ to $\exp(\frac{c}{c} x))=\exp(x)=u.$
\end{proof}

\section{ Words in two letters on  $\PSL(2,\BC)$}\label{section2}

In this section we  consider words in two letters. We
provide the explicit formulas for
$w(x,y),$  where $x,y$ are upper triangular matrices.
This    enables to extract some additional information on the image of  words
in two letters.
.

Consider a word map $w(x,y)=x^{a_1}y^{b_1}\dots x^{a_k}y^{b_k},$
where $a_i\ne 0 $ and $b_i\ne 0 $ {for all }$ i=1,...,k.$
Let  $A(w)=\sum_{i=1}^k a_i$, $B(w)=\sum_{i=1}^k b_i$. Let
 $ w:\tilde G^2\to \tilde G$
 be the induced  word map on $  G=\SL(2,\BC).$

If $A(w)=B(w)=0,$ then  $w\in F^{(1)}=[F,F].$ Since $F^{(1)}$ is a free group generated
by elements $w_{n,m}=[x^n,y^m], \ n\ne 0, \ m\ne 0$ (\cite{Ser}, Chapter 1, \S 1.3), the word $w$ with $A(w)=B(w)=0$
 may be written as a (noncommutative)  product (with $s_i\ne 0$)
 \begin{equation}\label{s1}
w=\prod_{1}^r w_{n_i,m_i}^{s_i}.
\end{equation}

Moreover, the shortest (reduced) representation of this kind is unique.
We denote by $S_w(n,m)$ the number of appearances  of $w_{n,m}$ in
representation \eqref{s1} of $w$ and by $R_w(n,m)$ the sum of exponents at
 all the appearances.    We denote by $Supp(w)$ the set of all pairs $(n,m)$ such that $w_{n,m}$ appears in the product. For example, if $w=w_{1,1}w_{2,2}^5w_{1,1}^{-1},$  then
$$Supp(w)=\{(1,1),(2,2)\}; S_w(1,1)=2, S_w(2,2)=1,$$ $$R_w(1,1)=0, R_w(2,2)=5.$$

The subgroup  $$F^{(2)}=[F^{(1)},F^{(1)}]=\{w\in F^{(1)} |R_w(n,m)=0 \  \text {for all }(n,m)\in Supp(w)\}.$$

\begin{example}\label{engel} The Engel word  $e_n=\underbrace{
[...[x,y],y],...y]}_{n \quad times}$ belongs to $F^{(1)}\setminus F^{(2)}$ (see also \cite{ET}).


Indeed,  the direct computation shows that
\begin{equation}\label{en1}
yw_{n,m}=yx^ny^mx^{-n}y^{-m}=yx^ny^{-1}x^{-n}\cdot x^nyy^mx^{-n}y^{-m}y^{-1}\cdot y=
w_{n,1}^{-1}w_{n,m+1}y,\end{equation}

\begin{equation}\label{en2}
 yw_{n,m}^{-1}=y\cdot y^mx^{n}y^{-m}x^{-n}=y^{(m+1)}x^{n}y^{-(m+1)}x^{-n}\cdot
x^nyx^{-n}y^{-1}\cdot y=w_{n,m+1}^{-1}w_{n,1}y.\end{equation}

 It follows that
\begin{equation}\label{en3}
yw_{1,m}^{s}y^{-1}=(w_{1,1}^{-1}w_{1,m+1})^s.\end{equation}

Let us prove by  induction that  $|R_{e_n}(1,n)|=1,\   S_{e_n}(1,n)=1$ and $ S_{e_n}(r,m)=0$
 if  $r\ne 1$ or $m>n,$  i.e.

\begin{equation}\label{en4}
e_n=(\prod_{1}^s w_{1,m_i}^{s_i})w_{1,n}^\varepsilon(\prod_{1}^t w_{1,k_j}^{t_j})
\end{equation}
for some integers $t\ge 0,  \  s\ge 0,   \ m_i<n,  \   k_j<n,$ and  where $\varepsilon=\pm 1.$

Indeed $e_1=w_{1,1}.$ Assume that the claim is valid for all $k\le n.$
We have $e_{n+1}=e_nye_n^{-1}y^{-1}$.
Using \eqref{en4},  we get
\begin{equation}\label{en5}
e_{n+1}=e_n(\prod_{t}^1 yw_{1,k_j}^{-t_j}y^{-1})yw_{1,n}^{-\varepsilon} y^{-1}(\prod_{s}^1y w_{1,m_i}^{-s_i}y^{-1}).
\end{equation}
Applying \eqref{en3} to every factor of this product, we obtain that $e_{n+1}$ has the needed form.
Thus the claim will remain to be valid for $n+1.$

Since $|R_{e_n}(1,n)|=1, $   \  $e_n\not\in F^{(2)}
.$\end{example}

\bigskip

Let us take
\begin{equation}\label{s4}
x=\begin{pmatrix} \lambda & c\\0 & \frac{1}{\lambda}\end{pmatrix},\end{equation}

\begin{equation}\label{s5}
y=\begin{pmatrix} \mu & d\\0 & \frac{1}{\mu}\end{pmatrix},\end{equation}

Then \begin{equation}\label{s6}
x^n=\begin{pmatrix} \lambda^n & c\cdot h_{|n|}(\lambda)sgn (n)\\0 & \frac{1}{\lambda^n}\end{pmatrix},\end{equation}

\begin{equation}\label{s7}
y^m=\begin{pmatrix} \mu^m & d\cdot h_{|m|}(\mu)sgn (m)\\0 & \frac{1}{\mu^m}\end{pmatrix},\end{equation}

Here $sgn$ is the $signum$ function, and (see \cite{BG}, Lemma 5.2) for $n\ge 1$

\begin{equation}\label{hn}
h_n(\zeta)=\frac{\zeta^{2n}-1}{\zeta^{n-1}(\zeta^{2}-1)}.
\end{equation}.

Note that $h_n( 1)=n.$

Direct computations show that

\begin{equation}\label{tt24}
x^{n}y^{m}=\begin{pmatrix} \lambda^{n}\mu^{m}& d\cdot \lambda^{n}sgn(m)h_{|m|}(\mu)
+c\cdot sgn(n)h_{|n|}(\lambda)\mu^{-m}
\\0 &\lambda^{-n}\mu^{-m} \end{pmatrix}.\end{equation}

\begin{equation}\label{tt25}
x^{-n}y^{-m}=\begin{pmatrix} \lambda^{-n}\mu^{-m}& -d\cdot \lambda^{-n}sgn(m)h_{|m|}(\mu)
-c\cdot sgn(n)h_{|n|}(\lambda)\mu^{m}
\\0 &\lambda^{n}\mu^{m} \end{pmatrix}.\end{equation}

 \begin{equation}\label{wnm}
 w_{n,m}(x,y)=\begin{pmatrix} 1 & f(c,d,n,m)\\0 &1 \end{pmatrix},\end{equation}
 where

  \begin{equation}\label{s10}
 f(c,d,n,m)=ch_{|n|}(\lambda)sgn (n)\lambda^n(1-\mu^{2m})+
 dh_{|m|}(\mu)sgn (m)\mu^m(\lambda^{2n}-1).
\end{equation}
Hence,
 \begin{equation}\label{s21}
w(x,y)=\prod_{1}^r w_{n_i,m_i}^{s_i}(x,y)=\begin{pmatrix} 1 & F_w(c,d,\lambda,\mu)\\0 &1 \end{pmatrix},\end{equation}

where $$F_w(c,d,\lambda,\mu)=\sum_1^r s_i f(c,d,n_i,m_i)=c\Phi_w(\lambda,\mu)+d\Psi_w(\lambda,\mu)$$ and
\begin{equation}\label{s22}
\Phi_w(\lambda,\mu)=\sum_{(\al,\be)\in Supp(w)}  R_w(\al,\be)sgn(\al)(1-\mu^{2\be})
\frac{(\lambda^{2|\al|}-1)\lambda^{\al}}{\lambda^{|\al|-1}(\lambda^{2}-1)},
\end{equation}

\begin{equation}\label{s23}
\Psi_w(\lambda,\mu)=\sum_{(\al,\be)\in Supp(w)}   R_w(\al,\be)sgn(\be)(\lambda^{2\al}-1)
\frac{(\mu^{2|\be|}-1)\mu^{\be}}{\mu^{|\be|-1}(\mu^{2}-1)}.
\end{equation}

(Since the order of factors in $w$ is  not relevant for \eqref{s22} and \eqref{s23} ,
we use  here $\al,\be$ instead of $n_i, m_i$ to simplify the formulas).

\begin{Proposition}\label{unipotent}
Rational functions $\Phi(\lambda,\mu)$ and  $\Psi(\lambda,\mu)$
are  non-zero   linearly   independent   rational functions. \end{Proposition}

\begin{remark} It is evident from the Magnus Embedding  Theorem that at least one of functions
$\Phi(\lambda,\mu)$ and  $\Psi(\lambda,\mu)$ is not identical zero. It follows from \propref{unipotent} that the same is valid for both of them.\end{remark}

\begin{proof}

\begin{Lemma}\label{fipsi}
If $\Phi_w(\lambda,\mu)\equiv 0$ then $ R_w(\al,\be)=0$ for all
$(\al,\be)\in Supp(w).$
\end{Lemma}
\begin{proof}

We use induction by the number $|Supp(w)|$ of elements  in $Supp(w)$  for the word $w.$
If $Supp(w)$ contains only one pair $(\al,\be)$, then there is nothing to prove, because
$$\Phi(\lambda,\mu)=R_w(\al,\be)h_{|\al|}(\lambda)sgn (\al)\lambda^\al(1-\mu^{2\be}).$$
Assume that for words $v$ with $|Supp(v)|=l $ it is proved. Let $w$ be such a word that  $|Supp(w)|=l+1. $

Let $n:=max\{\al \ |(\al,\be)\in Supp(w)\}.$

{\bf Case 1.} $n>0.$

We have

$$\Phi_w(\lambda,\mu)=\sum_{(\al,\be)\in Supp(w)} R_w(\al,\be) sgn(\al)(1-\mu^{2\be})
\frac{(\lambda^{2|\al|}-1)\lambda^{\al}}{\lambda^{|\al|-1}(\lambda^{2}-1)}=$$
$$\sum_{(\al,\be)\in Supp(w)} R_w(\al,\be)sgn(\al)(1-\mu^{2\be})\lambda^{a-|a|+1}(1+\lambda^{2}+\dots+
\lambda^{2(|\al|-1)}).$$

It means that the coefficient of $ \lambda^{2|n|-1}$ in  rational function   $\Phi_w(\lambda,\mu)$
is  $$p(\mu)=\underset{(n,\be)\in Supp(w)} \sum R_w(n,\be)(1-\mu^{2\be}).$$

Hence, if $\Phi_w(\lambda,\mu)\equiv 0,$  then  $p(\mu) \equiv 0,$  and all $R_w(n,\be)=0$  for all
 $\be.$

That means that
$\Phi_w(\lambda,\mu)=\Phi_v(\lambda,\mu),$ where $v$ is such a word that
may be obtained from
$w(x,y)=\prod_{1}^r w_{n_i,m_i}^{s_i}(x,y)$
by taking away every appearance of $w_{n,\be}:$

$$v=\underset{n_i\ne n}{\prod_{1}^r }w_{n_i,m_i}^{s_i}(x,y).$$

But $|Supp(v)|\leq l$ and by induction assumption $ R_v(\al,\be)=0$ for all
$(\al,\be)\in Supp(v).$  Thus Lemma is valid for $w$ in this case.

{\bf Case 2.}  $n<0.$ Let $-n':=min\{\al \ |(\al,\be)\in Supp(w)\}.$  We  proceed as in Case 1 with $-n'$ instead
 of $n:$
the coefficient of $\lambda^{-2n'+1}$ is $q(\mu)=\underset{(-n',\be)\in Supp(w)} \sum R_w(-n',\be)(1-\mu^{2\be}).$ If $\Phi_w(\lambda,\mu)\equiv 0,$  then  $q(\mu) \equiv 0,$  and all $R_w(-n',\be)=0$  for all
 $\be.$  Once more, we may replace $w$ by a word $v$ with  $|Supp(v)|\leq l.$
\end{proof}

  Clearly, the similar statement is valid for $\Psi_w(\lambda,\mu).$

The functions $\Phi$ and $\Psi$ are linearly independent, because $\Phi$
is  odd with respect to $\lambda$ and even with respect to $\mu,$
while  $\Psi$   has opposite properties.

\end{proof}

\begin{Proposition}\label{-id} Assume that the word $w\in F^{(1)}\setminus F^{(2)}$
and  that $\Phi_w(1,i)\ne 0 ,$ where $i^2=-1.$ Then $-id\in w_G, $  where $G=\SL(2,\BC).$\end{Proposition}

\begin{proof} Assume that  $\Phi(1,i)\ne 0 .$  From \eqref{s22} we get:

\begin{equation}\label{s222}
\Phi_w(1,i)=\sum_{(\al,\be)\in Supp(w), \be \  odd}  2R_w(\al,\be)\al.
\end{equation}

Take $$x=\begin{pmatrix}a & 0 \\ 0 & a^{-1}\end{pmatrix}$$
$$y=\begin{pmatrix}0 & 1 \\ -1 & 0\end{pmatrix}$$

Then $$[x,y]=\begin{pmatrix}a^2 & 0 \\ 0 & a^{-2}\end{pmatrix}$$
Thus, if

$$w=\prod_{1}^r w_{n_j,m_j}^{s_j},$$

then $$w(x,y)=\prod_{m_j \ odd}\begin{pmatrix}a^{2n_js_j} & 0 \\ 0 & a^{-2n_js_j}\end{pmatrix}=\begin{pmatrix}a^{N} & 0 \\ 0 & a^{-N}\end{pmatrix},$$

where $N=2\sum\limits_{m_j \ odd}n_js_j=\Phi_w(1,i)\ne 0.$

Choose $a$ such that $a^N=-1.$  Then $w(x,y)=-id.$\end{proof}

\begin{remark} The case $\Psi(i,1)\ne 0$  may be treated in the similar way, one should  only exchange roles of $x$ and $y.$\end{remark}

\begin{remark}   Let
$$w=\prod_{1}^r w_{n_j,m_j}^{s_j},$$
let $gcd(m_j)=k=2^ds, \ s$ odd.  Put $\mu_j=\frac{m_j}{k}$ and
$$u=\prod_{1}^r w_{n_j,\mu_j}^{s_j}.$$  Note that some of $\mu_j$ are  odd.
Let $z\in \SL(2,\BC)$ be such that
$$z^k=y=\begin{pmatrix}0 & 1 \\ -1 & 0\end{pmatrix}.$$
Then  $w(x,z)=u(x,y),$ hence, if  $\Phi_u(1,i)\ne 0 ,$  then $-id \in w_G.$\end{remark}

\section {Surjectivity on $\SL(2,\BC)$}\label{section3}

We keep the notation of \secref{section2}.

\begin{Lemma}\label{Ane0} Assume that $w=x^{a_1}y^{b_1}\dots x^{a_k}y^{b_k}, $  \
 \ $a_i\ne 0,  \  b_i\ne 0,   \ i=1,...,k$
$A=\sum a_i\ne 0$ or $B=\sum b_i\ne 0$ and  $x, y $
are  defined by \eqref{s4}, \eqref{s5} respectively.
Then
\begin{equation}\label{r21}
w(x,y)=\begin{pmatrix} \lambda^A\mu^B & \tilde F_w(c,d,\lambda,\mu)\\0 & \lambda^{-A}\mu^{-B}\end{pmatrix},\end{equation}

where $$ \tilde F_w(c,d,\lambda,\mu)=c \tilde\Phi_w(\lambda,\mu)+d \tilde \Psi_w(\lambda,\mu)$$ and
\begin{equation}\label{r22}
\tilde \Phi_w(\lambda,\mu)=\sum_{1}^k sgn(a_i)h_{|a_i|}(\lambda)\frac{\lambda^{\sum_{j<i}a_j}
\mu^{\sum_{j<i}b_j}}{\lambda^{\sum_{j>i}a_j}
\mu^{\sum_{j\ge i}b_j}},\end{equation}

\begin{equation}\label{r23}
\tilde \Psi_w(\lambda,\mu)=\sum_{1}^k sgn(b_i)h_{|b_i|}(\mu)\frac{\lambda^{\sum_{j\le i}a_j}
\mu^{\sum_{j<i}b_j}}{\lambda^{\sum_{j>i}a_j}
\mu^{\sum_{j> i}b_j}}.\end{equation}\end{Lemma}

\begin{proof}
We use induction on  the complexity $k$ of the word $w.$
Using \eqref{tt24},  we get

\begin{equation}\label{r24}
x^{a_1}y^{b_1}=\begin{pmatrix} \lambda^{a_1}\mu^{b_1}& d\cdot \lambda^{a_1}sgn(b_1)h_{|b_1|}(\mu)
+c\cdot sgn(a_1)h_{|a_1|}(\lambda)\mu^{-b_1}
\\0 &\lambda^{-a_1}\mu^{-b_1} \end{pmatrix}.\end{equation}

 Thus for $k=1$ the Lemma is valid. Assume that it is valid for $k'<k.$
Let $u=x^{a_1}y^{b_1}\dots x^{a_{k-1}}y^{b_{k-1}}$ and $w=ux^{a_k}y^{b_k}.$

By induction assumption,
$$u(x,y)=\begin{pmatrix} \lambda^{A-a_k}\mu^{B-b_k} &\tilde  F_u(c,d,\lambda,\mu)\\0 & \lambda^{-A+a_k}\mu^{-B+b_k}\end{pmatrix}.$$

From \eqref{tt24} we get
$$x^{a_k}y^{b_k}=\begin{pmatrix} \lambda^{a_k}\mu^{b_k}& d\cdot \lambda^{a_k}sgn(b_k)h_{|b_k|}(\mu)
+c\cdot sgn(a_k)h_{|a_k|}(\lambda)\mu^{-b_k}
\\0 &\lambda^{-a_k}\mu^{-b_k} \end{pmatrix}.$$

Multiplying matrices $u$ and $x^{a_k}y^{b_k}$
we get
$$\tilde F_w(c,d,\lambda,\mu)= \lambda^{A-a_k}\mu^{B-b_k} (d\cdot \lambda^{a_k}sgn(b_k)h_{|b_k|}(\mu)$$
$$+c\cdot sgn(a_k)h_{|a_k|}(\lambda)\mu^{-b_k} )+ \tilde   F_u(c,d,\lambda,\mu)\lambda^{-a_k}\mu^{-b_k}.$$

Thus,  the   induction assumption implies  that
$$\tilde \Phi_w(\lambda,\mu)=sgn(a_k)h_{|a_k|}(\lambda)\mu^{-b_k}\lambda^{A-a_k}\mu^{B-b_k}
+\sum_{1}^{k-1} sgn(a_i)h_{|a_i|}(\lambda)\frac{\lambda^{\sum_{j<i}a_j}
\mu^{\sum_{j<i}b_j}}{\lambda^{\sum_{j=i+1}^{k}a_j}
\mu^{\sum_{j=i}^{k}b_j}}$$

$$=\sum_{1}^k sgn(a_i)h_{|a_i|}(\lambda)\frac{\lambda^{\sum_{j<i}a_j}
\mu^{\sum_{j<i}b_j}}{\lambda^{\sum_{j>i}a_j}
\mu^{\sum_{j\geq i}b_j}}.$$

$$\tilde \Psi_w(\lambda,\mu)= sgn(b_k)h_{|b_k|}(\mu)\lambda^{a_k}\lambda^{A-a_k}\mu^{B-b_k}
+\sum_{1}^{k-1} sgn(b_i)h_{|b_i|}(\mu)\frac{\lambda^{\sum_{j\le i}a_j}
\mu^{\sum_{j<i}b_j}}{\lambda^{\sum_{j=i+1}^{k}a_j}
\mu^{\sum_{j=i+1}^{k}b_j}}$$

$$=\sum_{1}^k sgn(a_i)h_{|a_i|}(\lambda)\frac{\lambda^{\sum_{j\le i}a_j}
\mu^{\sum_{j<i}b_j}}{\lambda^{\sum_{j>i}a_j}
\mu^{\sum_{j> i}b_j}}.$$

\end{proof}

Denote: $$A_i=\sum_{j\le i}a_i;   \ B_i=\sum_{j<i}b_i,$$ and  let    $C$ be a curve
 $$C=\{
\lambda^A\mu^B =-1\}\subset \BC^2_{\lambda, \mu}. $$

Multiplying
 \eqref{r22} and \eqref{r23} by $\lambda^{A}\mu^B$ we see that on  $C$ the following relations are valid:
\begin{equation}\label{t1}
\tilde \Phi_w(\lambda,\mu) \bigm|_C=-\sum_{1}^k sgn(a_i)h_{|a_i|}(\lambda)\lambda^{2A_i-a_i}\mu^{2B_i},
\end{equation}
\begin{equation}\label{t2}
\tilde \Psi_w(\lambda,\mu)\bigm|_C=-\sum_{1}^k sgn(b_i)h_{|b_i|}(\mu)\lambda^{2A_i}
\mu^{\sum 2B_i+b_i}.\end{equation}

In particular,  on $C$

\begin{equation}\label{t11}
\tilde \Phi_w(1,\mu)\bigm|_C=-\sum_{1}^k a_i\mu^{2B_i},
\end{equation}
\begin{equation}\label{t21}
\tilde \Psi_w(\lambda,1)\bigm|_C=-\sum_{1}^k b_i\lambda^{2A_i}.
\end{equation}

\begin{Lemma}\label{noton} Assume that $A\ne 0$
and the word map  $w$ is not surjective. Then
$$\sum_{1}^k b_i\gamma^{2A_i}=0$$
 for every root $\gamma$ of equation
 $$q(z):=z^A+1=0.$$

If  $B\ne 0$  and  the word map  $w$ is not surjective, then
 $$\sum_{1}^k a_i\delta^{2B_i}=0$$ for every root $\delta$ of equation
 $$p(z):=z^B+1=0.$$ \end{Lemma}

 \begin{proof}

The matrices $z$ with $tr(z)=2 $ are in the image because $w(x,id)=x^A, $ $ w(id,y)=y^B.$
Assume now that for $K\ne 0$ the matrices

\begin{equation}\label{k3}\begin{pmatrix} -1 & K\\0 & -1\end{pmatrix}
\end{equation}
are not in the image. That implies that $\tilde \Phi_w(\lambda,\mu)\equiv 0$ and
$\tilde \Psi_w(\lambda,\mu)\equiv 0$
on the defined above curve   $$C=\{
\lambda^A\mu^B =-1\}\subset \BC^2_{\lambda, \mu}.$$

  If  $A\ne 0$  or    $B\ne 0,$  then, respectively, the pairs $(\gamma,1)$ and $(1,\delta)$  belong to the curve $C.$  We have to use only  \eqref{t11}, \eqref{t21}, respectively .\end{proof}

\begin{corollary}\label{B} Let $2B_i=k_iB+T_i,$ where $k_i$ are integers and  $0\le T_i<B\ne 0.$ If $w$ is not surjective, then for every $0\le T<B$

\begin{equation}\label{t111}
\sum_{i:T_i=T} a_i(-1)^{k_i}=0.
\end{equation}\end{corollary}

  \begin{proof} Indeed in this case $$0=\sum_{1}^k a_i\delta^{2B_i}=\sum_{T=0}^{B-1}\delta^{T}(\sum_{i:T_i=T} a_i(-1)^{k_i})$$ for any root $\delta$ of equation
 $$p(z)=z^B+1=0.$$
Since   $p(z)$  has no multiple roots,  it implies  that  $p(z)$ divides the polynomial

 $$p_1(z):=\sum_{T=0}^{B-1}z^{T}(\sum_{i:T_i=T} a_i(-1)^{k_i}).$$

 But since degree of $p(z)$ is bigger than degree of $p_1(z)$
 that can be only if $p_1(z)\equiv 0.$
 \end{proof}

\begin{corollary}\label{positive}
(\corref{1.4})
If all $b_i$ are positive, then the word map  $w$ is either surjective
or the square of another word $v\ne id$.\end{corollary}
 \begin{proof}
 In this case  $0\le 2B_i<2B$ and sequence $B_i$ is increasing.
If $w$ is not surjective, $p_1(z)\equiv 0$
 by \corref{B}.
 Thus for every
$B_i$ there is $B_j$ such that $2B_i=2B_j+B $ and $a_i-a_j=0.$

Thus, the sequence of $2B_i$ looks like:

$$0=2B_1, \  2b_1=2B_2,  \ 2(b_1+b_2)=2B_3, \dots,  2(b_1+\dots+b_s)=2B_{s+1}=B,$$
$$ 2(b_1+\dots+b_{s+1})=2B_{s+2}=B+2B_2=B+2b_1, $$
 $$2(b_1+\dots+b_{s+2})=2B_{s+3}=B+2B_3=B+2b_1+2b_2,
\dots,$$
$$2(b_1+\dots+b_{2s-1})=2B_{2s}=2B_s+B, $$   $$2(b_1+\dots+b_{2s})= 2B_{2s+1} =B+2B_{s+1}=2B.$$

It follows that  $k=2s$  and

 $$b_{s+1}=B_{s+2}-B_{s+1}=B_{2}-B_{1}=b_1;$$
$$b_{s+2}=B_{s+3}-B_{s+2}=B_{3}-B_{2}=b_2;$$
$$b_{2s-1}=B_{2s}-B_{2s-1}=B_{s}-B_{s-1}=b_{s-1};$$
$$b_k=b_{2s}=B_{2s+1}-B_{2s}=B_{s+1}-B_s=b_s.$$

Thus, $$ b_i=b_{i+s}, \ i=1,\dots, s, \ 2B_i=2B_{i+s}+B, a_i=a_{i+s}.$$

  Therefore the word is the square of $v=x^{a_1}y^{b_1}\dots x^{a_s}y^{b_s}.$
 \end{proof}

\begin{corollary}\label{negative}
If all $b_i$ are negative, then the word map of the word $w$ is either surjective
or the square of another word $v\ne id.$\end{corollary}
 \begin{proof}
We may change $y$ to $z=y^{-1}$ and apply \corref{positive} to the word $w(x,z).$
 \end{proof}

\begin{corollary}\label{A}
If all $a_i$ are positive, then the word map of the word $w$ is either surjective
or the square of another word $v\ne id.$\end{corollary}

 \begin{proof} Consider $v=x^{-1},$ \ $z=y^{-1},$  a word
$$w'(z,v)=w(x,y)^{-1}=y^{-b_k}x^{-a_k}\dots y^{-b_1}x^{-a_1}=z^{b_k}v^{a_k}\dots
z^{b_1}v^{a_1},$$
 and apply \corref{positive} to the word $w'(z,v).$
 \end{proof}

\section {Trace criteria of almost surjectivity}\label{section4}

For every word map $w(x,y):G^2\to G$ defined are  the trace polynomials $P_w(s,t,u)=tr(w(x,y))$ and $Q_w=tr(w(x,y)y)$ in three variables $s=tr(x), \  t=tr(y),  and
\ u=tr(xy).$
(\cite {FK}, \cite {Go}, \cite {Vo}).

In other words,  the maps
 $$\varphi_w:G^2\to G^2, \ \varphi_w(x,y)=(w(x,y),y)$$ and
$$\psi_w:\BC^3_{s,t,u}\to \BC^3_{s,t,u}, \ \psi_w(s,t,u)=(P_w(s,t,u),t, Q_w(s,t,u))$$
may be included into the following commutative
 diagram:

\begin{equation}\label{diagram}
\begin{CD}
G \times G   @>{\vp}>>   G \times G   \\
@V\pi VV @V\pi VV \\
\BC^3_{s,t,u} @>{\psi}>> \BC^3_{s,t,u}
\end{CD}.
\end{equation}

Moreover, $\pi$ is a surjective map (\cite{Go}). For details, one can be referred to
(\cite {BGKJ},\cite {BGaK}) .

Since the coordinate $t$ is invariant under $\psi,$ for every fixed value $t=a\in\BC$
we may consider the restriction $\psi_a(s,u)=(P_w(s,a,u),Q_w(s,a,u))$ of morphism $\psi_w$
onto the plane $\{t=a\}=\BC^2_{s,u}.$

\begin{Definition}
We say that  $\psi_a(s,u)$ is  {\bf Big} if the image $\psi_a(\BC^2_{s,u})=\BC^2_{s,u}\setminus T_a,$  where $T_a$ is  a finite set. We say that  the trace map $\psi_w$ of a word  $w\in F$ is {\bf Big}  if
there is a  value $a$ such that $\psi_a(s,u)$ is  {\bf Big}.
\end{Definition}

\begin{Proposition}\label{tracecriteria}    If   the trace map $\psi_w$ of a word  $w\in F$
is   {\bf Big} then the word map $w:G^2\to G$  is almost surjective. \end{Proposition}

\begin{proof}  Let $a$ be such a value of $t$ that the map $\psi_a$ is  {\bf Big}. Let $S_a=T_a\cup \{(2,a)\}\cup \{(-2,-a)\}$.
Consider a line  $C_+=\{s=2\}$ and $C_- =\{s=-2\}\subset\BC^2_{s,u}.$  Let $B_+=C_+\setminus(C_+\cap S_a);$
 $B_-=C_-\setminus(C_-\cap S_a).$  Since $S_a$ is finite, $
B_+\ne\emptyset, B_-\ne\emptyset.$  Moreover, since these curves are outside $S_a, $ we have: $D_+=\psi^{-1}(B_+)\ne\emptyset, \ D_-=\psi^{-1}( B_-)\ne\emptyset.$

Take  $(s_0,u_0)\in D_+$ and  $(s_1,u_1)\in D_-.$ Then $\psi_w(s_0,a,u_0)=(2,a,b)$ with $a\ne b;$
and $\psi_w(s_1,a,u_1)=(-2,a,d)$ with $a\ne -d.$    Projection $\pi:  G^2\to \BC^3_{s,t,u}$ is surjective, thus
there is a pair $(x_0,y_0)\in G^2$ such that $tr(x_0)=s_0, \ tr(y_0)=a,\  tr(x_0y_0)=u_0. $  Then $\pi(w(x_0,y_0))=\psi_w(s_0,a,u_0)=(2,a,b).$
Hence, $tr(w(x_0,y_0))=2,$ but $w(x_0,y_0)\ne id,$   since    $ tr(w(x_0,y_0)y_0)=b\ne a=tr(y_0).$
Similarly, there is a pair  $(x_1,y_1)\in G^2$ such that $tr(x_1)=s_1, \ tr(y_1)=a,\  tr(x_1y_1)=u_1. $  Then $\pi(w(x_1,y_1))=\psi_w(s_1,a,u_1)=(-2,a,d).$
Hence, $tr(w(x_1,y_1))=-2,$ but $w(x_1,y_1)\ne -id,$   since    $ tr(w(x_1,y_1)y_1)=d\ne -a=-tr(y_1).$

It follows that all the elements $z\ne-id $ with trace $2$ and $-2$  are in the image of the word map $w.$ \end{proof}

\begin{Corollary}\label{sequence}Assume that the trace map $\psi_w$  of a word $w$ is {\bf Big}. Consider a sequence of words defined recurrently in the following way:
$$v_1(x,y)=w(x,y); \
v_{n+1}(x,y)=w(v_n(x,y),y);$$
Then the word map $v_n:G^2\to G$ is almost surjective for all $n\ge1.$\end{Corollary}

\begin{proof} The trace map $\psi_n=\psi_{v_n}$ of the word map $v_n$ is the $n^{th}$ iteration $\psi^{(n)}_1$ of the
 trace map $\psi_1=\psi_w$ (see \cite{BGKJ} or \cite {BGaK}).  Let us show by induction, that all the maps $\psi_{n}$ are {\bf Big}.
Indeed $\psi_1$ is {\bf Big} by assumption, hence ${(\psi_1)}_a (\BC^2_{s,u})= \BC^2_{s,u}-T_a$ for some value $a$
and some finite set $T_a.$
 Assume now that  $\psi_{n-1}$  is {\bf Big}. Let for a  value $a$ of $t$
the image ${(\psi_{n-1})}_a(\BC^2_{s,u}) =\BC^2_{s,u}\setminus N$ for some finite set $N.$
Hence $${(\psi_n)}_a(\BC^2_{s,u})={(\psi_1)}_a({(\psi_{n-1})}_a(\BC^2_{s,u}))={(\psi_1)}_a(\BC^2_{s,u}\setminus N)\supset$$
$$\supset{(\psi_1)}_a(\BC^2_{s,u})\setminus{(\psi_1)}_a(N)=\BC^2_{s,u}\setminus( T_a\cup{(\psi_1)}_a(N)).$$
Thus ${(\psi_n)}_a$ is {\bf Big}  as well for the same value $a.$

According to \propref{tracecriteria}, the word map $v_n$ is almost surjective. \end{proof}

\begin{Example} Consider the word $w(x,y)=[yxy^{-1},x^{-1}] $ and the corresponding sequence
$$v_n(x,y)=[yv_{n-1}y^{-1},v_{n-1}^ {-1}].$$ This is one of the sequences  that were
 used for characterization of finite solvable groups (see  \cite{BWW}, \cite{BGKJ}, \cite {BGaK}).

We have ( \cite{BGKJ}, section 5.1)
$$tr(w(x,y))=f_1(s,t,u)=(s^2+t^2+u^2-ust-4)(t^2+u^2-ust)+2;$$
$$tr(w(x,y)y)=f_2(s,t,u)=f_1t+(s(st-u)-t)(s^2+t^2+u^2-ust-4)-t;$$

We want to show that for a general value  $t=a$ the system of equations
\begin{equation}\label{big1}
f_1(s,a,u)=A\end{equation}
\begin{equation}\label{big2}
f_2(s,a,u)=B\end{equation}
has solutions for all pairs $(A,B)\in \BC^2\setminus T_a,$  where $T_a$ is a finite set.

Consider the system

\begin{equation}\label{big3}
h_1(s,u,a,C):=(s^2+a^2+u^2-usa-4)(a^2+u^2-usa)=A-2:=C,\end{equation}
\begin{equation}\label{big4}
h_2(s,u,a,D):= (s(sa-u)-a)(s^2+a^2+u^2-usa-4)=B- a(C+1):=D.\end{equation}

Note that the leading coefficient with respect $u$  in  $h_1$ is $1$, in $h_2$ is $s.$
The Magma computations show that the resultant (elimination of $u$ ) of  $h_1-C$ and $h_2-D$
is of the form
$$R(s,a,C,D)=s^4p_1(a,C,D)+s^2p_2(a,C,D)+p_3(a,C,D).$$

It has a non-zero root $s\ne 0$ at any point $(a,C,D),$
where  at least two  of three polynomials  $p_1,p_2,p_3$ do not vanish.
MAGMA computation show that
the ideals  $J1=<p_1,p_2>\subset \BQ[a,C,D],$ \  $J2=<p_1,p_3>\subset \BQ[a,C,D],$  \  $J3=<p_2,p_3>\subset \BQ[a,C,D]$   generated,  respectively,  by $p_1(a,C,D)$ and $p_2(a,C,D), $  by $p_1(a,C,D)$ and $p_3(a,C,D), $ by $p_2(a,C,D)$ and $p_3(a,C,D),$    are  one-dimensional.
It follows that for a general value of $a$ the set

$\{p_1(a,C,D)=p_2(a,C,D)=0\}$

 \hskip 3cm     $\cup\{p_1(a,C,D)=p_3(a,C,D)=0\}$

 \hskip 5cm   $\cup\{p_2(a,C,D)= p_3(a,C,D)=0\}$
\newline  is a finite
subset $N_a\subset \BC_{C,D}.$  On the other hand, at any point  $(C,D) $ outside $N_a$ polynomial
$R_a(s)=R(s,a,C,D)$ has a  non-zero root, and, therefore system \eqref{big3}, \eqref{big4} has a solution.
Thus, outside  the  finite set of points $T_a=\{(A=C+2, B=D+a(C+1)) \ |(C,D)\in N_a\}\subset \BC_{A,B},$
system \eqref{big1}, \eqref{big2} has a solution as well.
Thus, $\psi_w=(f_1,t,f_2)$ is {\bf Big} and all the word maps $v_n$ are almost surjective on $G.$

Let us cite the Magma computations  for $t=a=1, $  where $p=h_1-C$ and $q=h_2-D.$
$R$ is the resultant of $p,q$ with respect to $u.$

 \begin{verbatim}
 > r:=u^2+s^2+1-u*s;
>
> p:=(r-4)*(r-s^2)-C;
>
> q:=(r-4)*(s*(s-u)-1)-D;
>
> R:=Resultant(p,q,u);
> R;
-s^4*C^3 - 2*s^4*C^2*D + s^4*C^2 - 2*s^4*C*D^2 + s^4*C*D
 - s^4*D^3 + s^4*D^2 + 4*s^2*C^2*D - 4*s^2*C^2 + 8*s^2*C*D^2
 - 6*s^2*C*D + 6*s^2*D^3 - 8*s^2*D^2 +
    C^2 - 2*C*D^2 + 8*C*D + D^4 - 8*D^3 + 16*D^2
>

>
> p1:=-C^3 - 2*C^2*D + C^2 - 2*C*D^2 + C*D - D^3 + D^2;
> p2:= 4*C^2*D - 4*C^2 + 8*C*D^2 - 6*C*D + 6*D^3 - 8*D^2;
> p3:=C^2 - 2*C*D^2 + 8*C*D + D^4 - 8*D^3 + 16*D^2;
> Factorization(p1);
[
    <C + D - 1, 1>,
    <C^2 + C*D + D^2, 1>
]
> Factorization(p2);
[
    <C^2*D - C^2 + 2*C*D^2 - 3/2*C*D + 3/2*D^3 - 2*D^2, 1>
]
> Factorization(p3);
[
    <C - D^2 + 4*D, 2>

\end{verbatim}

Clearly   every pair among polynomials  $p1,p2,p3$  has only finite number of common zeros.
For example,
$p_1=p_3=0$ implies
$D^2-5D+1=0$ or $(D^2-4D)^2+(D^2-4D)D+D^2=0.$

 Computations show also that the word $w(x,y)$ takes on value $-id.$  For example,
one make take $$x=\begin{pmatrix}- 1&1\\-2 &1\end{pmatrix}, \  y=\begin{pmatrix} 1&t\\0&1\end{pmatrix},$$
where $t^2=-1/2.$
Here are computations:
\begin{verbatim}
> R<t>:=PolynomialRing(Q);
>  X:=Matrix(R,2,2,[-1,1,-2,1]);
>  Y:=Matrix(R,2,2,[ 1,t,0,1]);
>   X1:= Matrix(R,2,2,[1,-1,2,-1]);
>  Y1:=Matrix(R,2,2,[1,-t,0,1]);
>
>  Z:=Y*X*Y1;
>
>    p11:=Z[1,1];
>    p12:=Z[1,2];
>    p21:=Z[2,1];
>     p22:=Z[2,2];
>
>  Z1:=Matrix(R,2,2,[p22,-p12,-p21,p11]);
>
>  W:=Z*X1*Z1*X;
>
>    q11:=W[1,1];
>   q12:=W[1,2];
>  q21:=W[2,1];
>   q22:=W[2,2];
>
>
>    q11;
16*t^4 + 8*t^3 + 12*t^2 + 4*t + 1
>   q12;
-8*t^4 - 4*t^2
>  q21;
16*t^3 + 8*t
>   q22;
-8*t^3 + 4*t^2 - 4*t + 1
\end{verbatim}
Therefore, $t^2=-1/2$ implies that $q_{11}=q_{22}=-1, \ q_{12}=q_{21}=0.$

\end{Example}

\section {The word  $v(x,y)=[[x,[x,y]],[y,[x,y]]]$}

In this section we provide an example of a   word $v$ that is surjective
though it belongs to  $F^{(2)}.$
The interesting feature of this word is the following: if we consider it as a polynomial in the
Lie algebra $\mathfrak{sl}_2$, ($[x,y]$ being  the Lie bracket)
then it is not surjective (\cite{BGKP},
Example 4.9).

\vskip 0.5 cm

\begin{theorem}\label{sur-psl}  The word $v(x,y)=[[x,[x,y]],[y[x,y]]]$ is surjective on $\SL(2,\BC)$ (and, consequently, on $\PSL(2,\BC)$). \end{theorem}
\begin{proof}
As it was shown in \propref{semisimple}, for every  $z\in  \SL(2,\BC)$
with $tr(z)\ne\pm2$ there are $x,y\in  \SL(2,\BC)^2$ such that $v(x,y)=z.$

Assume now that $a=\pm 2$. We have to show that $-id$ is in the image and that
 there are matrices
$x,y$ in $\SL(2,\BC), $  such that $$v(x,y):=\begin{pmatrix} q_{11} &q_{12}\\q_{21} & q_{22}\end{pmatrix}$$
has  the following
properties  :

\begin{itemize}
\item $q_{12}+q_{22}=\pm 2;$
\item $q_{12}\ne0.$
\end{itemize}

 We may look for these pairs among the matrices $x=\begin{pmatrix} 0 &b\\c & d\end{pmatrix}$
and $y=\begin{pmatrix} 1 &t
 \\0 & 1\end{pmatrix}.$

In the following MAGMA calculations
$C=[x,y],$ $D=[[x,y],x],$ $B=[[x,y],y],$ $A=[D,B].$

Ideal $I$ in the polynomial ring $\BQ[b,c,d,t]$ is defined by conditions $det(x)=1, tr(A)=2.$
Ideal $J$ in the polynomial ring $\BQ[b,c,d,t]$ is defined by conditions $det(x)=1, tr(A)=-2.$
Let  $T_+\subset \SL(2)^2 $ and  $T_-\subset \SL(2)^2$ be, respectively, the  corresponding affine
 subsets
 in affine variety $\SL(2)^2.$

The computations show that $q_{12}(b,c,d,t)$ does not vanish identically on $T_+$ or $T_-.$

 \begin{verbatim}

 >  Q:=Rationals();
> R<t,b,c,d>:=PolynomialRing(Q,4);
> X:=Matrix(R,2,2,[0,b,c,d]);
> Y:=Matrix(R,2,2,[ 1,t,0,1]);
>  X1:= Matrix(R,2,2,[d,-b,-c,0]);
> Y1:=Matrix(R,2,2,[1,-t,0,1]);
> C:=X*Y*X1*Y1;
>  p11:=C[1,1];
>  p12:=C[1,2];
>  p21:=C[2,1];
>  p22:=C[2,2];
> C1:=Matrix(R,2,2,[p22,-p12,-p21,p11]);
> D:=C*X*C1*X1;
>
>
> d11:=D[1,1];
> d12:=D[1,2];
> d21:=D[2,1];
> d22:=D[2,2];
> D1:=Matrix(R,2,2,[d22,-d12,-d21,d11]);
>
> B:=C*Y*C1*Y1;
>
>
> b11:=B[1,1];
> b12:=B[1,2];
> b21:=B[2,1];
> b22:=B[2,2];
> B1:=Matrix(R,2,2,[b22,-b12,-b21,b11]);
>
> A:=D*B*D1*B1;
>
> TA:=Trace(A);
>
> q12:=A[1,2];
> I:=ideal<R|b*c+1,TA-2>;
>
> IsInRadical(q12,I);
false
> J:=ideal<R|b*c+1,TA+2>;
>
> IsInRadical(q12,J);
false
>
\end{verbatim}

It follows that  the function $q_{12}(b,c,d,t)$
does not vanish identically on the   sets $T_+$ and $T_-$, hence,
 there are pairs with $tr(v(x,y))=2,v(x,y)\ne id,$
and $tr(v(x,y))=-2,v(x,y)\ne  - id.$

In order  to produce  the explicit solutions for $v(x,y)=-id$ and $v(x,y) =z, z\ne -id, tr(z)=-2, $
consider the following matrices depending on one parameter $d$:
$$x=\begin{pmatrix} 1-d&1\\-\frac{2}{3}&d\end{pmatrix},$$
$$y=\begin{pmatrix}2-3d &0\\0&3d-1\end{pmatrix}.$$

Since images of the commutator word on $\GL(2,\BC)$ and $\SL(2,\BC)$ are the same, we do not
require that $det(x)=1$ or $det(y)=1. $
 We only assume that $det(x)=d^2-d-2/3\ne 0$ and $det(y)=-9d^2+9d^2-2\ne 0.$

Let   $$A=v(x,y):=\begin{pmatrix} q_{11}(d) &q_{12}(d)\\q_{21}(d) & q_{22}(d)\end{pmatrix}$$
and $TA=tr(A)$.
Magma computations show that
$$ q_{11}(d)+1=N_{11}(d^2 - d + 1/3)H_{11}(d),$$
$$ q_{22}(d)+1=N_{22}(d^2 - d + 1/3)H_{22}(d),$$
$$q_{21}(d)=N_{21}(d-2/3)^2(d - 1/2)^3(d - 1/3)^2(d^2 - d - 2/3)(d^2 - d + 1/3)H_{21}(d),$$
$$q_{12}(d)=N_{21}(d-2/3)^2(d - 1/2)^3(d - 1/3)^2(d^2 - d - 2/3)(d^2 - d + 1/3)H_{12}(d),$$
$$TA+2=N(d^2 - d + 1/3)H(d),$$
where $N_{ij}$ and $N$ are non-zero rational numbers;  $H_{ij}$ and $H$ are  polynomials with rational coefficients that are
irreducible over $\BQ$ .
Moreover $deg H_{21}=deg H_{12}=25, \ deg H=38.$
It follows that  if $d^2 - d + 1/3=0 $  then $A=-id.$
If $d$ is  a root of $H$ that is  not a
root of $H_{21},$  then $A$ is a minus unipotent.

\end{proof}


\begin{thebibliography}{MMMM}



\bibitem{BG} T. Bandman, S. Garion,  {\em
 Surjectivity and equidistribution of the word $x^ay^b$ on $\PSL(2,q)$ and $\SL(2,q),$}
International Journal of Algebra and Computation (IJAC),{\bf 22}(2012), n.2, 1250017--1250050.







\bibitem {BGG} T. Bandman, S. Garion, F. Grunewald, {\em On the Surjectivity
of Engel Words on $\PSL(2,q)$}, Groups Geom. Dyn. {\bf 6} (2012), no. 3, 409--439.

\bibitem {BGaK} T. Bandman, S. Garion,B. Kunyavskii, {\em Equations in simple matrix groups: algebra, geometry, arithmetic, dynamics}, Cent. Eur. J. Math. {\bf 12} (2014), no. 2, 175--211.

\bibitem{BGKP}T. Bandman, N. Gordeev, B. Kunyavskii, E.  Plotkin, {\em Equations in simple Lie algebras}, J. Algebra {\bf 355} (2012), 67--79.

\bibitem{BGKJ}
T. Bandman, F. Grunewald, B. Kunyavskii, N. Jones, {\em Geometry and
arithmetic of verbal dynamical systems on simple groups}, Groups
Geom. Dyn. {\bf 4}, no.\ 4, (2010), 607--655.








\bibitem{Bo}
A. Borel, {\em On free subgroups of semisimple groups}, Enseign.
Math. (2), {\bf 29} (1983), no. 1-2, 151--164.

\bibitem{Borel}
A. Borel, Linear Algebraic Groups, 2nd edition. Springer-Verlag, New York, 1991.

\bibitem{Bourbaki} N. Bourbaki, General Topology, Chapters 1--4. Springer-Verlag, Berlin Heidelberg New York, 1989.
\bibitem
{BWW}
Bray J. N., Wilson J. S., Wilson R. A., {\em A characterization of finite soluble groups by laws in two
variables}, Bull. London Math. Soc., 2005, {\bf  37},179--186.




\bibitem{Ch1} P. Chatterjee,  {\em On the surjectivity of the power maps of algebraic groups in characteristic
zero}, Math. Res. Lett. {\bf9 } (2002) 741--756.
\bibitem{Ch2} P. Chatterjee, {\em On the surjectivity of the power maps of semisimple algebraic groups}, Math.
Res. Lett. {\bf 10} (2003) 625--633.

\bibitem{ET}  A. Elkasapy, A. Thom, {\em
About Got\^{o}'s method showing surjectivity of word maps}, Indiana Univ. Math. J., 63 (2014), no. 5, 1553–-1565.




\bibitem{Fr}
R. Fricke, {\em \"Uber die Theorie der automorphen Modulgruppen},
Nachr. Akad. Wiss. G\"ottingen (1896), 91--101.

\bibitem{FK}
R. Fricke, F. Klein, {\em Vorlesungen der automorphen Funktionen},
vol.~1--2, Teubner, Leipzig, 1897, 1912.






\bibitem{Go}
W. Goldman, {\em Trace coordinates on Fricke spaces of some simple hyperbolic surfaces}, Handbook of Teichmuller theory. Vol. II, 611--684, IRMA Lect. Math. Theor. Phys., 13, Eur. Math. Soc., Zurich, 2009;
{\em An exposition of results of Fricke and Vogt},
preprint available at
http://www.math.umd.edu/\~{}wmg/publications.html .

\bibitem{KBKP}    A. Kanel-Belov, B. Kunyavskii, E. Plotkin, {\em Word equations in simple groups and polynomial equations in simple algebras},
 Vestnik St. Petersburg Univ.: Mathematics {\bf 46 } (2013), no. 1, 3--13.




\bibitem{KKMP}  E.Klimenko,   B. Kunyavskii, J. Morita, E. Plotkin, {\em Word maps in Kac-Moudy settings},
preprint, ArXiv:0156.01422, (2015).

\bibitem{Ku} B. Kunyavskii, {\em  Complex and real geometry
of word equations in simple matrix groups and algebras},
 Preprint, 2014,  private communication.










\bibitem{MagnusCrelle} W. Magnus, {\em\"Uber den Beweis des Hauptideal Satzes}. J. reine angew. Math.  {\bf  170} (1934), 235--240.

\bibitem{McNinch}  G. J. McNinch,  {\em Optimal  $\SL(2)$-homomorphisms}. Comment. Math. Helv. {\bf 80} (2005),  391--426.



\bibitem{Se}
D. Segal, {\em Words: notes on verbal width in groups}, London
Mathematical Society Lecture Note Series {\bf 361}, Cambridge
University Press, Cambridge, 2009.



\bibitem{Ser} J.-P. Serre, {\em Trees}, Springer Monographs in Mathematics,
Springer-Verlag, Berlin, 2003.

\bibitem{TaYu} P. Tauvel, R.W.T.  Yu, Lie algebras and algebraic groups. Springer-Verlag, Berlin Heidelberg, 2005.


\bibitem{Vo}
H. Vogt, {\em Sur les invariants fundamentaux des equations
diff\'erentielles lin\'eaires du second ordre}, Ann. Sci. E.N.S,
3-i\`eme S\'er. {\bf 4} (1889), Suppl. S.3--S.70.


\bibitem{Wehrfritz} B.A.F. Wehrfritz, {\em A residual property of free metabelian groups}. Arch. Math. {\bf 20} (1969), 248--250.

\bibitem{Wilson} J.S. Wilson, {\em Free subgroups in groups with few relations}. L'Enseignement Math. (2) {\bf 56} (2010), 173--185.

\end{thebibliography}
\end{document}